\documentclass[a4paper,12pt]{article}
\usepackage[utf8x]{inputenc}
\usepackage[english]{babel}
\usepackage{graphicx}
\usepackage{amscd}
\usepackage{amsfonts}
\usepackage{latexsym}
\usepackage{amssymb,amsmath}
\usepackage[usenames]{color}
\usepackage{caption}
\usepackage{amscd}
\usepackage{psfrag}

\usepackage{amsthm}

\newtheorem*{teormain}{Main Theorem}
\newtheorem*{teorA}{Theorem A}
\newtheorem*{teorB}{Theorem B}
\newtheorem{teor}{Theorem}[section]
\newtheorem{cor}[teor]{Corollary}
\newtheorem{prop}[teor]{Proposition}
 \newtheorem{lemma}[teor]{Lemma}
 
 \newtheorem{remark}[teor]{Remark}
\newtheorem{defi}[teor]{Definition}

\def\C{\mathbb{C}}

\def\F{\mathcal{F}}

\def\D{\mathbb{D}}
\def\d{\mathcal{D}}
\def\R{\mathbb{R}}

\def\Z{\mathbb{Z}}

\def\H{\mathbb{H}}

\def\l0{_{\lambda_0}}
\def\i0{_{\iota_0}}

\begin{document}
\title{Mating quadratic maps with the modular group II}

\author{Shaun Bullett 
\and 
Luna Lomonaco}

%

\maketitle

\begin{abstract}

In 1994 S. Bullett and C. Penrose introduced the one complex parameter family of $(2:2)$ holomorphic correspondences $\F_a$:
$$\left(\frac{aw-1}{w-1}\right)^2+\left(\frac{aw-1}{w-1}\right)\left(\frac{az+1}{z+1}\right)
+\left(\frac{az+1}{z+1}\right)^2=3$$
and proved that for every value of $a \in [4,7] \subset \R$ the correspondence $\F_a$ is a mating between a quadratic polynomial
$Q_c(z)=z^2+c,\,\,c \in \R$, and the modular group $\Gamma=PSL(2,\Z)$. They conjectured that this is the case for every member of the family $\F_a$ which has $a$ in the connectedness locus.

We show here that matings between the 
modular group and rational maps in the parabolic quadratic family $Per_1(1)$ provide a better model: we prove that every member of the family $\F_a$ which has $a$ in the connectedness locus is 
such a mating.
\end{abstract}
\setcounter{tocdepth}{1}
\tableofcontents

\section{Introduction}

The analogies between the iteration of holomorphic maps and the action of Kleinian groups were first 
enumerated by Sullivan in the mid 1980s.
His landmark paper \cite{S}, where he proved the conjecture of Fatou that there are no wandering domains 
for a rational map on the Riemann sphere, includes the first version of what it is 
now called Sullivan's dictionary, in which definitions, theorems and conjectures in the world of holomorphic
maps are related to analogous definitions, theorems and conjectures in the world of Kleinian groups.
Sullivan draws attention to deep parallels between the
 Fatou set $F_f$ and Julia set $J_f$ of a holomorphic map $f$ on $\widehat \C$, and the ordinary set $\Omega(G)$ and limit set $\Lambda(G)$ respectively
 of a finitely generated Kleinian group $G$ acting on $\widehat \C$, and his proof of the no wandering domains
theorem for rational maps is inspired by the method used to prove Ahlfors' finiteness theorem in the world of Kleinian groups.\\

  	  \begin{figure}
  	    	  \centering
	 \includegraphics[height= 5.5cm]{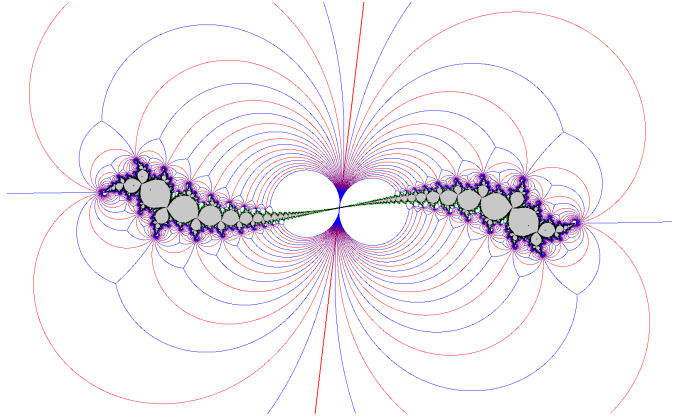}
 	  \caption{\small Limit set for $\F_a$, $a=4.565+0.420i$.}
	  \label{limset}
	 	  \end{figure}

Both rational maps and finitely generated Kleinian groups can be regarded as particular cases of correspondences.
 An $n$-to-$m$ holomorphic correspondence  on $\widehat \C$ is a multi-valued
 map $\F: z \rightarrow w$ defined by a
 polynomial relation $P(z,w)=0$. A rational map $f(z)= p(z)/q(z)$ becomes an $n$-to-$1$
correspondence defined by $P(z,w)=0$, where $P(z,w)= wq(z)-p(z)$, and any finitely generated Kleinian group 
$G$ with generators 
$$\gamma_j(z)=\frac{a_jz+b_j}{c_jz+d_j}$$ 
can be regarded as an $(n:n)$ correspondence by taking 
$$P(z,w)= \prod_{j=1}^n \left(w(c_jz+d_j)-(a_jz+b_j)\right). $$
For example, since 
$$\alpha(z)=z+1 \quad {\rm and} \quad \beta(z)=\frac{z}{z+1}$$ 
generate the modular group  $\Gamma=PSL(2,\Z)$, the orbits 
of $\Gamma$ on $\widehat \C$
are the orbits of the $(2:2)$
correspondence defined by 
$$(w-(z+1))(w(z+1)-z)=0.$$

The study of iterated holomorphic correspondences was initiated by Fatou \cite{F} in 1922, with an analysis 
of a family of examples `sur lesquels', he remarks, `on voit apparaitre d\'ej\`a certaines propri\'et\'es, assez  
diff\'erentes de celles auxquelles donnent lieu les cas d'it\'eration d\'ej\`a \'etudi\'es' [`in which one already sees
the appearence of certain properties somewhat different from those arising in the cases of iteration studied up till now']. He concludes his article with 
the comment that one may treat various examples of iterated algebraic functions in an analogous fashion, `mais une th\'eorie g\'en\'erale de ce probl\`eme ne parait pas facile. Nous pensons pouvoir y revenir ult\'erieurement'
[`but a general theory for this problem does not seem easy. We hope to return to this in the future'].
The next developments of which we are aware came in the 1990s, when McMullen and Sullivan in their foundational work \cite{MS}, defined a (one-dimensional) {\it holomorphic dynamical system} to be a collection
of holomorphic relations on a complex $1$-manifold, and developed a common framework in which rational maps, Kleinian groups and holomorphic correspondences can be treated simultaneously. At around the same time
researchers in integrable systems \cite {BV} were investigating the complexity of symmetric holomorphic correspondences associated with elliptic curves, a topic also prefigured in the introduction to Fatou's article.\\

Also in the 1990s, the first author and C. Penrose observed behaviour such as that
illustrated in Figure \ref{limset}, in a particular family of $(2:2)$ correspondences, 
namely the one parameter family $\F_a$ defined by 
$$\left(\frac{aw-1}{w-1}\right)^2+\left(\frac{aw-1}{w-1}\right)\left(\frac{az+1}{z+1}\right)+
\left(\frac{az+1}{z+1}\right)^2=3.$$
Computer plots appeared to show two copies (denoted in this article by $\Lambda_{a,-}$ and $\Lambda_{a,+}$) 
of the filled Julia set of a quadratic polynomial, together with an action of the modular group on the 
complement (denoted $\Omega_a$ here), prompting the question as to whether in the world of holomorphic correspondences there 
might exist `matings' between quadratic polynomials and the modular group. Bullett and Penrose \cite{BP}
constructed an abstract combinatorial mating between the modular group and any member of the quadratic 
family which has connected and locally connected filled Julia set (see Section \ref{mat}). 
Holomorphic correspondences realising these combinatorial matings are holomorphic realisations of
Minkowski's {\it question mark function} \cite{Min}, a homeomorphism from the unit interval to the positive real line
which sends a real number expressed in binary to a real number with a corresponding continued fraction expression. 
On the binary expression side of the mating is the Douady-Hubbard coding of rays for quadratic polynomials, which 
is key to combinatorial descriptions of Julia sets and renormalisation theory. On the continued fraction side, the 
action of the modular group is related to the generation of Farey sequences of rationals, and thence to the Riemann Hypothesis (we thank Charles Tresser for drawing our attention to the work of Franel \cite{Fr} and Landau \cite{Lan}, 
showing that the Riemann Hypothesis is equivalent to certain conditions concerning the uniformity of distribution 
of such sequences).\\

The main result of \cite{BP} is that
for $a$ in the real interval $[4,7]$ the holomorphic
correspondence $\F_a$ is indeed a mating between a (real) quadratic polynomial and the modular group. 
More generally, Bullett and Penrose conjectured that each 
$\F_a$ for which the parameter $a$ is in the connectedness locus for the family 
is a mating between a quadratic polynomial and the modular group. 
Their conjecture was subsequently proved for a large class of values of the
parameter $a$, by applying Haissinsky's technique of `pinching' to polynomial-like maps (see \cite{BH}). 
But the technique is not applicable for {\it all} values of $a$ in the connectedness locus, and we would argue
that the root cause of the difficulty is that the family of quadratic polynomials is the wrong model for the problem.
Whatever the value of $a$, the branch of $\F_a$ which fixes $z=0$ is parabolic, with multiplier at the 
parabolic fixed point equal to 1 (see Proposition \ref{power_series}).
This fact makes the use of polynomial-like mappings tricky and finally inefficient, and 
suggests that the optimum description of the correspondences $\F_a$  might be as matings 
between the modular group and members of some family of {\it parabolic} quadratic maps. 
As we shall demonstrate below, this is indeed the case, the family of maps being 
$$P_A:z \to z+\frac{1}{z}+A, \quad A \in \C,$$  which we recall are
the quadratic rational maps with a parabolic fixed point of multiplier $1$ at infinity and critical 
points at $\pm 1$.
Note that $P_{A'}$ is conformally conjugate to $P_{A''}$ if and only if $A'=\pm A''$; in Milnor’s notation the set of (conformal) conjugacy classes is denoted $Per_1(1)$.

	\begin{figure}
	 
 	\centering
 	\includegraphics[height= 5.5cm]{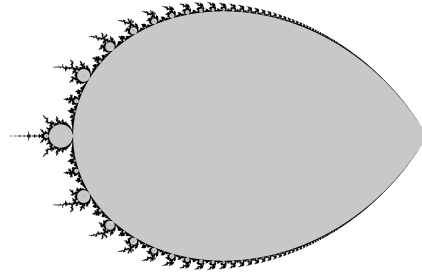}
 	\caption{\small The modular Mandelbrot set $M_{\Gamma}$}
 	\label{Mcorr}
 	\end{figure}
\begin{defi}\label{mat} We say that ${\mathcal F}_a$ is a mating between the rational quadratic map $P_A:z \to z+1/z+A$ 
and the modular group $\Gamma=PSL(2,{\mathbb Z})$ if


(i) the $2$-to-$1$ branch of ${\mathcal F}_a$ for which $\Lambda_{a,-}$ is invariant is hybrid equivalent 
to $P_A$ on $\Lambda_{a,-}$, and

(ii) when restricted to a $(2:2)$ correspondence from $\Omega_a$ to itself, ${\mathcal F}_a$ is conformally conjugate 
to the pair of M\"obius transformations $\{\alpha,\ \beta\}$ from the complex upper half plane ${\mathbb H}$ to itself. 

\end{defi}

Formal definitions of the sets $\Lambda_{a,-}$, $\Lambda_{a,+}$, and $\Omega_a$ are given in Section \ref{dynamics}. In the same Section we also define the {\it Klein combination locus} $\mathcal K\subset \C$ 
and the {\it connectedness locus} ${\mathcal C}_\Gamma\subset {\mathcal K}$ of the family of correspondences ${\mathcal F}_a$. Given these concepts, we are in a position to state the main result of this paper:

\begin{teormain}\label{mating}
For every $a\in {\mathcal C}_\Gamma$ the correspondence ${\mathcal F}_a$ is a mating between some rational map $P_A:z \to z+1/z+A$ and 
$\Gamma$.
\end{teormain}

The layout of this paper is as follows. In Section \ref{prelim} we assemble facts concerning {\it Fatou coordinates} 
and {\it parabolic-like mappings} that will be needed later.
In Section \ref{dynamics} we investigate some dynamical properties of the family $\F_a$ and in Section
\ref{thma} we prove:  
\begin{teorA}
 For every $a \in {\mathcal C}_\Gamma$,  when restricted to a $(2:2)$ correspondence from $\Omega_a$ to itself, ${\mathcal F}_a$ is conformally conjugate 
to the pair of M\"obius transformations $\{\alpha: z \to z+1,\ \beta: z \to z/(z+1)\}$ from the complex upper half plane ${\mathbb H}$ to itself. 
\end{teorA}

In Section \ref{thmb} we prove that every ${\mathcal F}_a$ with $a$ in the {\it Klein combination locus} $\mathcal K$ can be surgically modified to become a single-valued
{\it parabolic-like} map in the sense of \cite{L1} on a neighbourhood of the backward limit set 
$\Lambda_{a,-}$. Since this parabolic-like map can then be {\it straightened} (\cite{L1}) into 
a rational map of the form $P_A:z\to z+1/z+A$, we obtain the following:

\begin{teorB}\label{main}
For every parameter value $a \in {\mathcal K}$,  after a surgery supported outside the limit set,
the branch of $\F_a$ fixing $\Lambda_{a,-}$ restricts to a parabolic-like
mapping, and therefore on $\Lambda_{a,-}$ is hybrid equivalent to a member of the family $Per_1(1)$
of quadratic rational maps. 
\end{teorB}

Note that since the Julia set of a rational map is the closure of the set of repelling periodic points,
and quasiconformal maps preserve the nature 
of periodic points, the theorem implies the following:
\begin{cor}
 For each $a \in {\mathcal K}$, the boundary of $\Lambda_{a,-}$ is the closure of the set of 
repelling periodic points of the branch of $\F_a$ fixing $\Lambda_{a,-}$.
\end{cor}

The Main Theorem is a consequence of Theorems A and B.\\

As we shall see, the closed disc ${\mathcal D}=\{a:|a-4|\le 3\}$ is contained in the Klein combination locus ${\mathcal K}$
(apart from the point $a=1$, where $\F_a$ is undefined). Let $M_\Gamma$ denote the \textit{modular Mandelbrot set} 
${\mathcal C}_\Gamma \cap {\mathcal D}$. This set $M_\Gamma$ (Figure \ref{Mcorr}), which we believe to 
be the whole of ${\mathcal C}_\Gamma$, and which visibly resembles the classical Mandelbrot set, 
was first plotted in \cite{BP}.
In \cite{BL1} we investigate the dynamics of the family $\F_a$, and in particular we prove a new inequality of 
Yoccoz type which has as consequences the facts that 
$M_\Gamma$ is contained in a lune within ${\mathcal D}$ of internal angle strictly less than $\pi$, and that
for all $a\in M_\Gamma$ the limit set $\Lambda_{a,-}$ is contained in a dynamical space lune of internal angle
strictly less than $\pi$.
This in turn will allow us in \cite{BL2} to undertake the surgery of Theorem B in a sufficiently
uniform manner to deduce that $M_\Gamma$ is homeomorphic to the connectedness locus $M_1$ of 
the family $Per_1(1)$. Together with the proof announced by Pascale Roesch and Carsten Petersen that 
$M_1$ is homeomorphic to the classical Mandelbrot set $M$, this will finally prove 
that $M_\Gamma$ is homeomorphic to $M$. 
\\

{\bf Acknowledgements.} The authors are very grateful to Adam Epstein for first suggesting to them 
in 2011 that parabolic-like mappings might be applied to the family of correspondences $\F_a$, and to 
Adam Epstein, Carsten Petersen and Pascale Roesch for very helpful discussions at various stages of 
the work reported on here. This research has been partially supported by EU Marie-Curie IRSES Brazilian-European partnership in Dynamical Systems (FP7-PEOPLE-2012-IRSES 318999 BREUDS) and the Funda\c{c}\~ao de amparo a pesquisa do estado de S\~ao Paulo (Fapesp, process 2013/20480-7).



\section{Preliminaries}\label{prelim}
This section is devoted to a summary of results we will use during the article.
\subsection{Petals and Fatou coordinates}
A holomorphic map $g(z)=z+bz^2+\ldots$, with $b\neq 0$, defined in a neighbourhood of the origin, has a parabolic fixed point at the origin with multiplicity 1. A complex number \textbf{v} points in the repelling direction if $b$\textbf{v} is real and positive, and a complex number 
\textbf{w} points in the attracting direction if $b$\textbf{w} is real and negative.
An open set in a neigbourhood of the origin is called an attracting petal if it is mapped into itself and if each orbit eventually absorbed by it converges to the origin from the attracting direction \textbf{v}. Similarly, a repelling petal is an open set contained in its image and with orbits escaping from the origin in the repelling direction \textbf{w}.

There exists a well-established body of knowledge concerning attracting and repelling petals at parabolic fixed
points of holomorphic functions $g$, and {\it Fatou coordinates} on these petals. We shall make use of 
petals with the properties listed in the following Proposition.


\begin{prop}\label{petals}
For every holomorphic function $g$ as above, and every angle 
$0<\theta<\pi$, inside every neighbourhood of $0$ there exists a repelling petal $U_\theta^+$ containing an open sector of angle $2\theta$ centered at the origin
and symmetric with respect to the repelling direction. Each of these petals is
equipped with a conformal homeomorphism $\Phi^+$ (known as a Fatou coordinate) from $U_\theta^+$ to a subset $V_\theta^+$
of the complex plane consisting of all points $w=u+iv$ to the left of some curve which has asymptotes $u= -|v|\cot(\theta)-c$, 
with $c$ large so that $|w|$ is large for all $w\in V_\theta^+$
(see \cite{M2}), with the following properties:

(i) $\Phi^+$ is a composition $\psi^{-1}\phi$ where $\phi(z)=1/(-bz)$ and $\psi$ (defined on $V_\theta^+$) is asymptotic to the identity, in the sense that $\lim_{|w|\to\infty}\psi(w)/w=1$ for all $w\in V_\theta^+$;

(ii) $\Phi^+$ conjugates $g^{-1}$ on $U_\theta^+$  to $w \to w-1$ on $V_\theta^+$.

\end{prop}

\begin{proof} This is an immediate consequence of Chapter 10 in \cite{M}, or Chapter 6.5 in \cite{Be}. For the estimate of the asymptotes
of the repelling petal in the $w$ coordinate, see the proof of Theorem 7.2 in \cite{M2}.
\qed
\end{proof}

An {\it attracting petal} $U_\theta^-$ is a repelling petal for $g^{-1}$, and has a Fatou coordinate conjugating
$g$ to $w \to w+1$ on the corresponding domain $V_\theta^+$ in $\C$. We observe that 
$U_\theta^+$ and $U_\theta^-$ are foliated by {\it invariant curves} (invariant under $g^{-1}$ and $g$ 
respectively), corresponding to the respective foliations of  
$V_\theta^+$ and $V_\theta^-$ by horizontal lines.  When $U_\theta^+\cap U_\theta^-$ is non-empty (which is always the case when $\theta>\pi/2$) the two foliations on the intersection will usually be different: nevertheless for both these foliations on the intersection, leaves which correspond to horizontal lines in the 
$w$-plane sufficiently far above or below the real axis, extend to invariant (under both $g$ and $g^{-1}$) topological 
circles, {\it horocycles}, in the $z$-plane.

\subsection{$Per_1(1)$ and parabolic-like maps}

Consider the family of quadratic rational maps having a
parabolic fixed point of multiplier $1$ at $\infty$. Normalizing by fixing the critical points at $\pm 1$, this family is
$$Per_1(1) = \{ P_A(z)= z+1/z+A \, | \, A \in \C \}/(A\sim -A)$$
For a map in $Per_1(1)$, denoting by $\Lambda$ the parabolic basin of attraction of infinity,
we can define the filled Julia set of $P_A$ to be
$K_A= \widehat \C \setminus \Lambda$ (the map $P_0(z)=z+1/z$ is the unique map in the family $Per_1(1)$ with two parabolic attracting
petals, and we set $K_0=\overline \H_l$).\\
 


A {\it parabolic-like map} is a map which behaves in a similar way to a member of the family $Per_1(1)$ in a neighbourhood of its filled Julia set. The definition extends the notion of a polynomial-like map to a map with a parabolic external class:

\begin{defi}\label{parabolic-like-map}
 A parabolic-like map is a 4-tuple ($f,U',U,\gamma$) where
\begin{itemize}
	\item $U',U$ are open subsets of $\C$, with $U',\,\, U$ and $U \cup U'$ isomorphic to a disc, and $U'$ not
contained in $U$,
	\item $f:U' \rightarrow U$ is a proper holomorphic map of degree $d$ with a parabolic fixed point at $z=z_0$ of
 multiplier 1,
	\item $\gamma:[-1,1] \rightarrow \overline {U}, \gamma(0)=z_0$ is an arc, forward invariant under $f$, $C^1$
on $[-1,0]$ and on $[0,1]$, and such
that
$$f(\gamma(t))=\gamma(dt),\,\,\, \forall -\frac{1}{d} \leq t \leq \frac{1}{d},$$
$$\gamma([ \frac{1}{d}, 1)\cup (-1, -\frac{1}{d}]) \subseteq U \setminus U',\,\,\,\,\gamma(\pm 1) \in \partial U.$$
It resides in repelling petal(s) of $z_0$ and it divides $U',U$ into $\Omega', \Delta'$ and $\Omega, \Delta$
respectively, such that $\Omega' \subset \subset U$
(and $\Omega' \subset \Omega$), $f:\Delta' \rightarrow \Delta$ is an isomorphism and
$\Delta'$ contains at least one attracting fixed petal of $z_0$. We call the arc $\gamma$ a \textit{dividing arc}.

\end{itemize}
\end{defi}

The filled Julia set of a parabolic-like map $(f,U',U,\gamma)$ is the set of points that never escape $\Omega'\cup \{z_0\}$, this is
$K_{f}=\{z\in U'\;\vert\; \forall n \geq 0,\,\,\ f^{n}(z)\in U' \setminus \Delta'  \} $, and the Julia set is defined as
$J_{f}:=\partial K_{f}$ (see \cite{L1}).
By the Straightening Theorem for parabolic-like maps, any degree $2$ parabolic-like map is hybrid equivalent to a member of the family
$Per_1(1)$, a unique such member if the filled Julia set is connected.
\section{Dynamics of $\F_a$}\label{dynamics}

We consider the family of $(2:2)$ holomorphic correspondences on the Riemann sphere which have 
the form ${\mathcal F}_a:z \to w$, where 
$$\left(\frac{az+1}{z+1}\right)^2+\left(\frac{az+1}{z+1}\right)\left(\frac{aw-1}{w-1}\right)
+\left(\frac{aw-1}{w-1}\right)^2=3$$
for a parameter $a\in {\mathbb C}$, $a\ne 1$. 
The reason for studying this particular family is the following lemma (the content of which is in \cite{BP}, 
repeated here to establish notation) together with Proposition 1.4 of \cite{BH}, which states that every 
mating between a quadratic map and the modular group which supports a {\it compatible involution} 
(see \cite{BH}) is conformally conjugate to a member of this family.

\begin{lemma}\label{factorization}
In the coordinate $Z=\frac{az+1}{z+1}$, the correspondence $\F_a$ is the composition $J\circ Cov_0^Q$ where
$$J(Z)=\frac{(a+1)Z-2a}{2Z-(a+1)}$$
is the involution which has fixed points $1$ and $a$, and $Cov_0^Q:Z \to W$ is the deleted covering correspondence of the rational map $Q(Z)=Z^3-3Z$.

%
%
\end{lemma}

\begin{proof}
Consider the map $Q(Z)=Z^3-3Z$. It has a double critical point at infinity and simple critical points at $\pm 1$, and up to pre- and post-composition 
 by M\"obius transformations, every degree $3$ rational map with exactly $3$ distinct critical points is equivalent to $Q(Z)$.
\\
 
Let $Cov^Q:Z \rightarrow W$ be the $(3:3)$ covering correspondence of $Q$, which is the 
correspondence exchanging the preimages of $Q$, or in other words acting on the fibres of $Q$. This is the correspondence defined by
$$Q(Z)=Q(W),$$
or more explicitly by
$$Z^3-3Z=W^3-3W.$$

Let $Cov^Q_0:Z \rightarrow W$ be the $(2:2)$ correspondence defined by
$$\frac{Q(Z)-Q(W)}{Z-W}=0,$$
that is,
$$Z^2+ZW+W^2=3.$$
This is called the \textit{deleted covering correspondence} of $Q$, since its graph is obtained from that of $Cov^Q$ by deleting
the graph of the identity. \\

Post-composing this last correspondence by the involution $W \rightarrow J(W)$ we obtain the $(2:2)$ correspondence 
defined by the polynomial
$$Z^2+Z (J(W)) + (J(W))^2=3.$$
This is the correspondence 
 $$Z^2+Z \left(\frac{(a+1)W-2a}{2W-(a+1)}\right)+ \left(\frac{(a+1)W-2a}{2W-(a+1)}\right)^2=3,$$
 which is, via the change of coordinates

$$Z=\frac{az+1}{z+1}, \ \ W=\frac{aw+1}{w+1},$$
the correspondence $\F_a$.
\qed
\end{proof}

Note that in the coordinate $z$, the involution $J$ becomes $z \leftrightarrow -z$. The choice of whether to work in the coordinate
$Z$ or in the coordinate $z$ depends on whether it is more convenient to have a simple expression for $Cov^Q_0$ or for $J$.
We will denote by $P$ the common fixed point of $Cov^Q_0$ and $J$ ($P$ is the point $Z=1$ or $z=0$ in our 
two coordinate systems).\\

By a {\it fundamental domain} for $Cov^Q_0$ we shall mean a maximal
open set which is disjoint from its image under $Cov^Q_0$. (In this article fundamental domains will 
always be open sets.)

\begin{defi}\label{Klein-comb} The {\it Klein combination locus} $\mathcal K$ for the family of 
correspondences ${\mathcal F}_a$ 
is the set of parameter values $a$ for which there
exist simply-connected fundamental domains $\Delta_{Cov}$ and $\Delta_{J}$ for $Cov^Q_0$ and $J$ 
respectively, bounded by Jordan curves, such that
$$\Delta_{Cov} \cup \Delta_{J} = \hat{\mathbb C}\setminus\{P\}.$$
We call such a pair of fundamental domains $(\Delta_{Cov},\Delta_{J})$ a {\it Klein combination pair}.

\end{defi}

\begin{defi}\label{standard-pair}
For $a$ in ${\mathcal D}=\{a:|a-4|\le 3\}$, the {\it standard pair of fundamental domains} is that given by taking
 $\Delta_{Cov}$ to be the region of the $Z$-plane ${\mathbb C}$ to the right of $Cov^Q_0((-\infty,-2])$, and 
$\Delta_{J}$ to be complement in $\hat{\mathbb C}$ of the closed round disc in the $Z$-plane 
$\hat{\mathbb C}$ which has centre 
on the real axis and boundary circle through the points $1$ and $a$. 
\end{defi}
	\begin{figure}
	
	 \centering
	  \psfrag{a}{\small $a$}
\psfrag{L}{\small $L'$}
 \psfrag{A}{\small $\overline \Delta_{J}$}
\psfrag{-2}{\small $-2$}
\psfrag{1}{\small $1$}
\psfrag{2}{\small $2$}
 \psfrag{C}{\small $\Delta_{Cov^Q_0}$}
 	\includegraphics[height= 5.5cm]{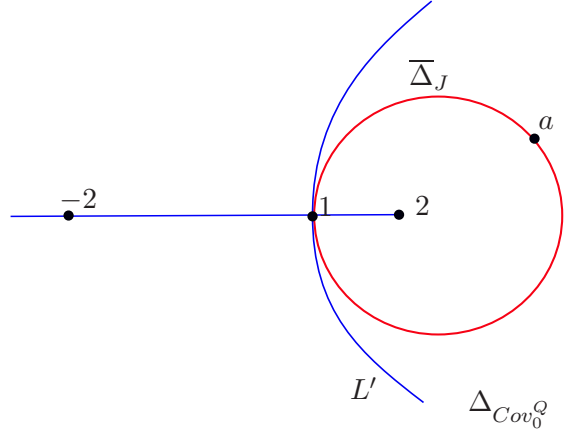}
 	\caption{\small Standard fundamental domains for $\F_a$.}
 	\label{standardfundamentaldomain}
 	
	\end{figure}

\begin{prop}\label{23}
For all $a \in {\d}$ (apart from the parameter value $a=1$ where the correspondence is undefined), the 
standard pair of fundamental domains is a Klein combination pair. Hence 
${\d}\setminus\{1\}\subset {\mathcal K}$.
\end{prop}
\begin{proof}
The real line interval $L=[-\infty,+2]$ has inverse image $Q^{-1}(L)$ the line interval $L$ itself, together with
a curve $L'$ which crosses $L$ orthogonally at $Z=1$ and runs off towards $\infty$ in directions approaching
angles $\pm\pi/3$ to the positive real axis (Fig. \ref{standardfundamentaldomain}). This line $L'$ is the image of $[-\infty,-2]$ under 
$Cov^Q_0$: an elementary computation shows that
$$L'=\{\left(1+\frac{t}{2}\right)\pm i \sqrt{3\left(t+\left(\frac{t}{2}\right)^2\right)}:t \in [0,\infty]\}.$$
Now the component of $\C\setminus L'$ which lies to the right of $L'$ is a fundamental domain 
for $Cov^Q_0$, that is to say it is a maximal
open set which is disjoint from its image under $Cov^Q_0$ (see also Example 1.2 in \cite{B}). But this 
component is our {\it standard} fundamental domain for $Cov^Q$ (Definition \ref{standard-pair}.)
\\

The standard $\Delta_{J}$ is self-evidently a fundamental domain for the involution $J$, so it only 
remains to verify that for $a\in \d \setminus\{1\}$, the domains $\Delta_{Cov}$
and $\Delta_{J}$ satisfy the Klein combination condition. However an elementary computation 
shows that $L'$ meets the circle which has centre $Z=4$ and radius $3$ at the single point $Z=1$. 
It follows that $\Delta_{Cov}\cup \Delta_{J}\supseteq \hat{\C}\setminus\{1\}$ for all $a\in {\d}\setminus\{1\}$.
\qed
\end{proof}

\begin{prop}\label{domains}
 For every $a \in \mathcal{K}$ and Klein combination pair $(\Delta_{Cov},\Delta_{J})$, the correspondence $\F_a$ has the following properties when its domain and co-domain are restricted as indicated: 
 
\begin{itemize}
 

\item
$\F_a^{-1}(\overline{\Delta}_J)\subset\overline{\Delta}_J$, and
$\F_a|: \F_a^{-1}(\overline{\Delta}_J) \to \overline{\Delta}_J$
is a (single-valued, continuous) $2$-to-$1$ map;
\item
$\F_a(\hat{\mathbb C}\setminus \Delta_J)\subset \hat{\mathbb C}\setminus \Delta_J$, and
$\F_a|: \hat{\mathbb C}\setminus \Delta_J \to \F_a(\hat{\mathbb C}\setminus \Delta_J)$
is a $1$-to-$2$ correspondence, conjugate via $J$ to $\F_a^{-1}|:\overline{\Delta}_J \to \F_a^{-1}(\overline{\Delta}_J)$.

\end{itemize}

\end{prop}
	\begin{figure}
	
	 \centering
	  \psfrag{a}{\small $a$}
\psfrag{A}{\small $\overline \Delta_{J}$}
\psfrag{D}{\small $D_J$}
\psfrag{fD}{\small $\F_a(D_J)$}
  \psfrag{2}{\small $(2:1)$}
 \psfrag{1}{\small $(1:2)$}
  \psfrag{3}{\small $(1:1)$}
\psfrag{F}{\small $\F_a^{-1}(\overline \Delta_{J})$}
\psfrag{e}{\small $\F_a^{-2}(\overline \Delta_{J})$}
 	\includegraphics[height= 6cm]{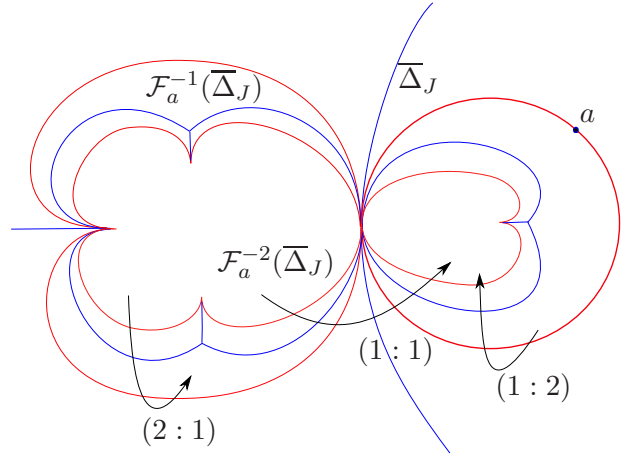}
 	\caption{\small Images and preimages of $\overline \Delta_{J}$ in the $Z$-coordinates.}
 	\label{imagepreimage}
 	
	\end{figure}
\begin{proof} 

From the Klein Combination condition (Definition \ref{Klein-comb}) we have that $\hat{\C}\setminus\Delta_J\subset\overline{\Delta}_{Cov}$ and 
$\hat{\C}\setminus \Delta_{Cov}\subset \overline{\Delta}_J$. Thus (see Figure \ref{imagepreimage}):
$$\F_a^{-1}(\overline{\Delta}_J)=Cov_0^Q\circ J(\overline{\Delta}_J)=Cov_0^Q(\hat{\C}\setminus\Delta_J)
\subset Cov_0^Q(\overline{\Delta}_{Cov})=\hat{\C}\setminus \Delta_{Cov}\subset \overline{\Delta}_J.$$
Now note that
$$Cov_0^Q|: Cov_0^Q(\Delta_{Cov}\cup\{P\}) \to \Delta_{Cov}\cup\{P\}$$
is  a (single-valued, continuous) $2$-to-$1$ map, and so the same is true for
$$\F_a=J\circ Cov_0^Q|:Cov_0^Q(\Delta_{Cov}\cup\{P\} ) \to J(\Delta_{Cov}\cup\{P\}).$$
Since $\overline{\Delta}_J\subset J(\Delta_{Cov}\cup\{P\})$ by the Klein Combination condition, and
also $\F_a^{-1}(\overline{\Delta}_J)=Cov_0^Q(\hat{\C}\setminus\Delta_J)\subset Cov_0^Q(\Delta_{Cov}\cup\{P\})$
by the same condition, it follows that 
$$\F_a|: \F_a^{-1}(\overline{\Delta}_J) \to \overline{\Delta}_J$$
is also a (single-valued, continuous) $2$-to-$1$ map.\\

As $J\circ \F_a=Cov_0^Q=\F_a^{-1}\circ J$ we deduce that
$$\F_a^{-1}|:J( \F_a^{-1}(\overline{\Delta}_J)) \to J(\overline{\Delta}_J)$$
is a $2$-to-$1$ map. But $J(\overline{\Delta}_J)=\hat{\C}\setminus\Delta_J$ and $J(\F_a^{-1}(\overline{\Delta}_J))=\F_a(\hat{\C}\setminus\Delta _J)$. Thus 
$$\F_a^{-1}|:\F_a(\hat{\mathbb C}\setminus \Delta_J) \to  \hat{\mathbb C}\setminus \Delta_J$$
is a $2$-to-$1$ map, and so its inverse 
$$\F_a|: \hat{\mathbb C}\setminus \Delta_J \to \F_a(\hat{\mathbb C}\setminus \Delta_J)$$
is a $1$-to-$2$ correspondence. Moreover this $1$-to-$2$ correspondence is conjugate, via $J$, to 
$\F_a^{-1}|:\overline{\Delta}_J \to \F_a^{-1}(\overline{\Delta}_J)$,  and it follows from
$\F_a^{-1}(\overline{\Delta_J})\subset \overline{\Delta_J}$ that 
$\F_a(\hat{\C}\setminus\Delta_J)\subset\hat{\C}\setminus\Delta_J$. \qed
\end{proof}
We next examine the behaviour of $\F_a$ around the fixed point $P$ ($Z=1$).
\begin{prop} \label{power_series}
Let $\zeta=Z-1$. 
When $a\ne 7$ the power series expansion of the branch of $\F_a$ which fixes $\zeta=0$ has the form:
$$\zeta \to \zeta + \frac{a-7}{3(a-1)}\zeta^2 + \ldots$$
and so the Leau-Fatou flower at the fixed point has a single attracting petal. 
When $a=7$  the expansion has the form:
$$\zeta \to \zeta  + \frac{1}{27}\zeta^4 + \ldots$$
and so the flower at the fixed point has three attracting petals.
\end{prop}

\begin{proof}
By Lemma \ref{factorization}, $\F_a=J\circ Cov^Q_0$, where $J$ is the involution which has fixed points $1$ and $a$:
$$J(Z)={{(a+1)Z-2a}\over{2Z-(a+1)}}$$
and $Cov^Q_0: Z \to W$ where $Z^2+ZW+W^2=3$. Therefore the branch of $Cov^Q_0$ fixing $Z=1$
is $Z \to W$ where
$$W= \frac{-Z+(12-3Z^2)^{1/2}}{2}.$$
Changing coordinates to $\zeta,\omega$ where $Z=\zeta+1$ and $W=\omega+1$, so that the fixed
point is at $\zeta=0$, this branch of $Cov_0^Q$ becomes:
$$\omega=-\frac{\zeta}{2}+\frac{3}{2}\left(\left(1-\frac{2\zeta}{3}-\frac{\zeta^2}{3}\right)^{1/2}-1\right)
=-\zeta-\frac{\zeta^2}{3}-\frac{\zeta^3}{9}-\frac{2\zeta^4}{27}-\ldots$$
In these coordinates the involution $J$ is: 
$$\zeta \to -\zeta\left(\frac{1}{1-\frac{2\zeta}{a-1}}\right)
=-\zeta-\frac{2\zeta^2}{a-1}-\frac{4\zeta^3}{(a-1)^2}-\frac{8\zeta^4}{(a-1)^3}-\ldots$$
Composing the two power series and collecting up terms we deduce that the branch of 
$\F_a=J\circ Cov_0^Q$ which fixes $\zeta=0$ sends $\zeta$ to:
$$ \zeta + \frac{a-7}{3(a-1)}\zeta^2 + \left(\frac{a-7}{3(a-1)}\right)^2\zeta^3+
\left(\frac{2}{27}-\frac{2}{3(a-1)} +\frac{4}{(a-1)^2}-\frac{8}{(a-1)^3}\right)\zeta^4+\ldots$$ 
completing the proof.
\qed
\end{proof}

For $a\ne 7$ there is a unique {\it repelling direction} at the parabolic fixed point. 
From Proposition \ref{power_series}, in the $\zeta$ coordinate this is the direction
$$\zeta = \frac{{\bar a}-7}{{\bar a}-1}.$$
For $a=7$, there are three repelling directions: $\zeta=0, e^{2\pi i/3}, e^{4\pi i/3}$.
\\

\begin{defi}
Let $P$ be the parabolic fixed point of our correspondence $\F_a$, $a\ne 7$. We call the line defined by the repelling direction the {\it parabolic axis} at $P$, 
and we say that a differentiable curve $\ell$ passing through $P$ is  {\it transverse to the parabolic axis} if $\ell$ crosses this axis at a non-zero angle. (For $a=7$ we adopt the convention that the `parabolic axis' is the real axis, in both the $Z$-coordinate and the $z$-coordinate.)
\end{defi}




\begin{cor}\label{horizontals}
For $a\ne 7$, given any smooth curve $\ell$ passing through $P$ transversely to the parabolic axis, 
there is a repelling petal  $U_\theta^+$ and Fatou coordinate $\Phi^+$ on $U_\theta^+$ such that $\Phi^+(\ell)$
(in the $w=u+iv$ plane)
intersects every horizontal leaf $v=c$ in $V_\theta^+=\Phi^+(U_\theta^+)$ which corresponds to a sufficiently
large value of $|c|$.
\end{cor}

\begin{proof} 
The line $\ell$ meets the repelling direction at $P$ at some angle $0<\alpha<\pi$. Choose $\theta$ with
 $\alpha<\theta<\pi$. By Proposition \ref{petals}, as we travel along $\ell$ towards $P$ from either side, the final 
part of our journey is contained in $U_\theta^+$. The result follows, since  $\Phi^+:U_\theta^+ \to V_\theta^+$ sends 
a line meeting the repelling direction at $P$ at angle $\alpha$ to a curve the points $w(t)$ of which have $\lim_{t\to \infty}|w(t)|=\infty$ 
and $\lim_{t\to\infty}\arg(w(t))=\pi -\alpha$.
\qed
\end{proof}

\begin{figure}
\centering
\psfrag{e}{\small $\phi$}
\psfrag{cov}{\tiny $\F^{-1}(\partial \Delta_{Cov})$}
\psfrag{dc}{\tiny $\partial \Delta_{Cov}$}
\psfrag{j}{\tiny $\partial \Delta_J$}
\psfrag{1}{\tiny $\ell_1$}
\psfrag{2}{\tiny $\ell_2$}
\psfrag{a}{\tiny $U_{\theta}^-$}
\psfrag{f}{\tiny $\Phi^-$}
\psfrag{e}{\tiny $V_{\theta}^-$}
	\includegraphics[width= 10cm]{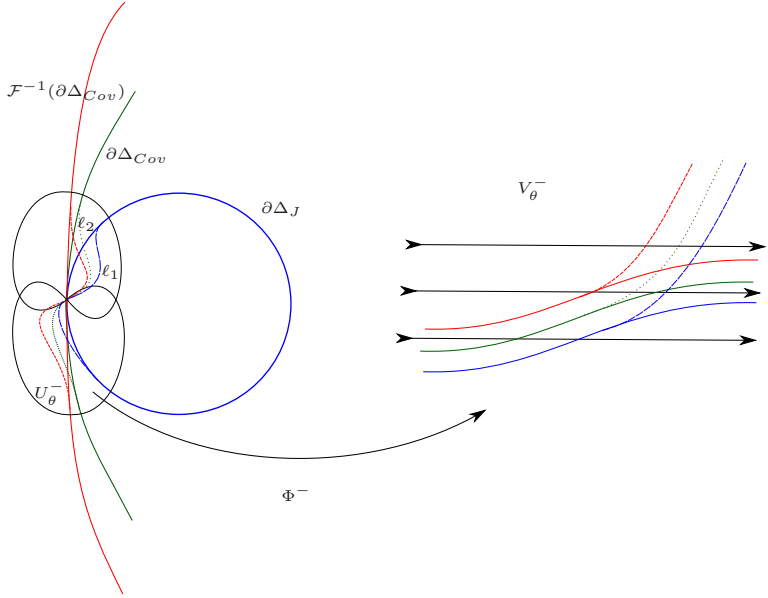}
\caption{\small Changing domains}
\label{modifdom}
\end{figure}

\begin{prop}\label{nice-domains} 
For $a \in {\mathcal K}$, we may always choose a Klein combination pair $(\Delta_{Cov},\Delta_{J})$ of
fundamental domains which have boundaries which are smooth at $P$ and transverse to the parabolic axis.
\end{prop}
\begin{proof}
By definition the Jordan curves bounding  $\Delta_{Cov}$ and $\Delta_{J}$ meet only at $P$. 
By making small perturbations to these curves if need be, we can ensure they are both smooth, apart from
an angle of $2\pi/3$ on $\partial\Delta_{Cov}$ at the double critical point ($Z=\infty)$ of $Q$. At $P$ the smooth 
curves $\partial\Delta_{Cov}$ and $\partial\Delta_J$ are tangent to one another (since the Klein combination 
condition excludes the possibility that they cross). For $a \in int({\d})$, that is $|a-4|<3$, the boundaries of the standard 
pair $(\Delta_{Cov},\Delta_{J})$ at their intersection $P$ ($Z=1)$ are parallel to 
the imaginary axis in the $Z$-plane, and as $a$ lies inside the circle in the $Z$-plane 
which has diameter the real interval $[1,7]$, we know that
$$arg\left(\frac{\bar{a}-7}{\bar{a}-1}\right)\ne \pm \frac{\pi}{2}$$ so the parabolic axis is tranverse
to the imaginary axis and we are done. When $a=7$, by our convention the parabolic axis is the real axis, 
which is transverse to the imaginary axis, so again we are done. 

\medskip
However for 
$a \in \partial \d\setminus \{7\}$ the boundaries of the standard pair are {\it tangent} 
to the parabolic axis, and so small horocycles at $P$ are tangent to $\partial \Delta_{J}$ there.
We shall see that in this situation, by making a small modification to the boundaries of the standard pair near $P$,
we can construct a new Klein combination pair which have boundaries transverse to the parabolic axis.
 More generally, for $a\in {\mathcal K}$ not necessarily in $\overline{\mathcal D}$, suppose we have Klein combination domains
$\Delta_J$ and $\Delta_{Cov}$ whose boundaries approach $P$ tangentially to the parabolic 
axis at $P$. Choose an angle $0<\theta<\pi/2$ and attracting and repelling petals $U_\theta^\pm$ which are 
sufficiently small that they do not intersect. Using the fact that the invariant foliations on these 
petals give us a complete picture of the dynamics of ${\mathcal F}_a$ on them, we can modify 
the part of $\partial\Delta_J$ which lies in the repelling petal
by replacing a small segment by a curve $\ell_1$ which approaches $P$ transversely to the parabolic axis 
and meets $Cov^Q_0(J(\ell_1))(={\mathcal F}_a^{-1}(\ell_1))$ only at the point $P$ (figure \ref{modifdom}). 
Next we modify $\partial\Delta_{Cov}$ on the same petal, replacing a segment with a 
curve $\ell_2$ lying between ${\mathcal F}_a^{-1}(\ell_1)$ and $\ell_1$. Finally on the attracting petal
we replace a segment of $\partial\Delta_J$ by $J(\ell_1)$ and a segment of $\partial\Delta_{Cov}$
by $Cov^Q_0(\ell_2)$. Since $Cov^Q_0$ acts on a neighbourhood of $P$ as an involution with 
fixed point $P$, rotating one side of figure \ref{modifdom} to the other, we see that $\ell_1\cup J(\ell_1)$ meets
$\ell_2\cup Cov^Q_0(\ell_2)$ only at $P$, and so we can use these as boundaries of modified fundamental
domains which still satisfy the Klein combination condition. 
\qed
\end{proof}

\begin{defi}\label{limitset} For $a \in {\mathcal K}$, with $(\partial\Delta_{Cov},\partial\Delta_J)$ chosen
with boundaries transverse to the parabolic axis at $P$, the forward limit set of ${\mathcal F}_a$ is defined 
to be 
$$\Lambda_{a,+}=\bigcap_{n=0}^\infty \F_a^n(\hat{\mathbb C}\setminus\Delta_{J}),$$
the backward limit set is defined to be
$$\Lambda_{a,-}=\bigcap_{n=0}^\infty \F^{-n}(\overline{\Delta}_{J})=J(\Lambda_{a,+})$$
and the limit set 
is defined to be
$\Lambda_a=\Lambda_{a,+}\cup\Lambda_{a,-}$, noting that by Proposition \ref{domains} we have $\Lambda_{a,+}\cap \Lambda_{a,-}=\{P\}$.
The regular set $\Omega_a$ is defined to be $\hat{\mathbb C}\setminus\Lambda_a$.
\end{defi}
Note that, by Proposition \ref{domains}, the sets $\Lambda_a$ and $\Omega_a$ are 
completely invariant under $\F_a$, and the involution $J$ conjugates $\F_a$ on $\Lambda_{a,-}$ to 
$\F_a^{-1}$ on $\Lambda_{a,+}$ 
(see also the fifth of the `Comments on Theorem 2' in \cite{B}).  
\begin{remark} The partition of  $\hat{\mathbb C}$ into $\Lambda$ and $\Omega$ is independent of the choice of Klein
combination domains, provided these domains have boundaries transverse to the parabolic axis at $P$.
For what can go wrong if we do not make this requirement, see Remark \ref{bad-domains} following 
the proof of Theorem A below. 
\end{remark}


\begin{defi} The connectedness locus for the family ${\mathcal F}_a$ 
is the subset ${\mathcal C}_\Gamma$ of $\mathcal K$
for which $\Lambda_{a,-}$, and hence also $\Lambda_{a,+}$ and $\Lambda_a$, is connected.
\end{defi}

Since $\Delta_{Cov} \cup \Delta_{J} = \hat{\mathbb C}\setminus\{P\},$
the proof of Theorem 2 in \cite{B}, which is a version for correspondences of the Klein Combination Theorem \cite{K,Mas}
(sometimes informally known as the `Ping-Pong Theorem'),
shows that ${\mathcal F}_a$ acts on $\Omega_a$ properly discontinuously (see the 4th point of Theorem 2 in \cite{B}) and faithfully
(since it acts freely on the set $\Omega'_a$ obtained from $\Omega_a$ by removing the grand orbit of fixed points of $J$ and $Cov_0^Q$), with fundamental 
domain $$\Delta=\Delta_{Cov} \cap \Delta_J.$$
(The theorem in \cite{B} is stated for correspondences $\F=Cov^P * Cov^Q$, where $P,\ Q$ are rational maps and $Cov^P$ and $Cov^Q$ are the covering correspondences. Writing $J(z)=-z$ and $P(z)=z^2$ we have $J=Cov_0^P$ and thus our $\F_a$ has the form 
$Cov_0^P\circ Cov_0^Q$. Note that if $Cov^P * Cov^Q$ acts freely on $\Omega'_a$, then $Cov^P_0 \circ Cov^Q_0$ acts faithfully on $\Omega_a$, where $Cov^P_0,\,Cov^Q_0$ are the deleted covering correspondences of $P$ and $Q$ respectively.)

\section{Proof of Theorem A}\label{thma}

We start by observing that by Definition \ref{limitset} and Proposition \ref{domains}, for every 
$a \in {\mathcal C}_{\Gamma}$ 
the regular set $\Omega= \Omega_a$
is open and simply connected, and therefore there exists a Riemann map 
$\phi: \Omega \rightarrow \H$. We will now prove that:
\begin{enumerate}
 \item there exist M\"obius transformations $\sigma$ of order $2$ and $\rho$ of order $3$, both in $PSL_2(\R)$, such that $\phi$
 conjugates $\F_a$ to  $\{\sigma\rho,\sigma\rho^{-1}\}$;
 \item the free product $\langle\sigma\rangle * \langle\rho\rangle$ is a faithful and discrete representation of $C_2*C_3$ in $PSL_2(\R)$;
 \item this representation is conjugate to $PSL_2(\Z)$.
\end{enumerate}

\textbf{Step 1.} 
Note that 
on a neighborhood of $\infty$ the B\"ottcher map conjugates the map $Q(Z)=Z^3-3Z$ to the map $Z \rightarrow Z^3$. It follows that on a neighbourhood of $Z=\infty$, the covering correspondence of $Q$ is conjugate to that of $Z \to Z^3$ via a homeomorphism $\hat\varphi$, say.
This can be extended to a conjugacy $\hat\varphi$ on the whole of $\Omega$, since the only critical point of $Q$ on $\Omega$ 
is the double critical point at $Z= \infty$. Thus  
$Cov^Q$ on $\Omega$ is conjugate via $\hat\varphi$ to $\{I, \hat \rho, \hat\rho^2 \}$ on some simply-connected open set 
$\Omega'\subset \hat{\C}$, where 
$\hat\rho(z)=e^{2\pi i/3}$, and so $Cov^Q_0$ is conjugate to $\{\hat\rho, \hat\rho^2 \}$. 
If $R: \Omega' \rightarrow \H$ is a Riemann map, $\phi: R\circ\hat\varphi: \Omega \rightarrow \H$ is a Riemann map conjugating the action of $Cov^Q_0$ on $\Omega$ to the action of an order $3$ rotation
$\rho$ on $\H$. 
On the other hand, since $J|_\Omega$ is an involution,  $J|_\Omega$ is conjugate by $\phi$ to some involution $\sigma$ on $\H$.
Therefore $\F_a= J \circ Cov^Q_0$ is conjugate by $\phi$ on $\Omega$ to  $\{\sigma\rho,\sigma\rho^{-1}\}$.

\textbf{Step 2.} By the correspondence ping-pong theorem (\cite{B}) we have that $\F_a$ acts on $\Omega$ faithfully and properly discontinuously.
Since $\phi$ is a homeomorphism, 
$\langle\sigma\rangle * \langle\rho\rangle$ also acts faithfully and properly discontinuously on $\H$. Therefore (since $\sigma$ is an involution and $\rho$ is an order $3$ rotation)
$\langle\sigma\rangle * \langle\rho\rangle$ is a faithful and discrete representation of $C_2*C_3$ in $PSL_2(\R)$.

\textbf{Step 3.}
To complete the proof we must prove that the representation of $C_2*C_3$ on ${\mathbb H}$ is the
standard representation as $PSL(2,{\mathbb Z})\subset PSL(2,{\mathbb R})$.
For every discrete representation of $C_2*C_3$ the orbifold ${\mathbb H}/(\langle\sigma\rangle * \langle\rho\rangle)$ 
is conformally isomorphic to a sphere with a $2\pi/3$-cone point, a $\pi$-cone point, and
either a single boundary component or a puncture point (a cusp is conformally equivalent to a neigbourhood of a puncture point). 
The representation is conjugate to $PSL_2({\mathbb Z})$ if and only if
the orbifold  ${\mathbb H}/(\langle\sigma\rangle * \langle\rho\rangle)$ has a puncture point.
Since $\phi$ is an isomorphism, ${\mathbb H}/(\langle\sigma\rangle * \langle\rho\rangle)$ is conformally equivalent to
$\Omega/\langle{\mathcal F}_a\rangle$.
By Proposition \ref{power_series} the point $P$ ($Z=1$) is a parabolic 
fixed point of ${\mathcal F}_a$. Let $(\Delta_{J},\Delta_{Cov})$ be a Klein combination pair with boundaries transverse to the parabolic axis (such a pair exists by Proposition \ref{nice-domains}). 
By Proposition \ref{petals} there exists a repelling petal $U_\theta^+$ containing all points of the line $\partial\Delta_{J}$ which lie sufficiently close to 
$P$, and by Corollary \ref{horizontals} the image of this line under the Fatou coordinate $\Phi^+$ meets all lines $v=c$ in the $w$-plane (where $w=u+iv$) which have $|c|$ sufficiently large. Writing $W$ for the intersection between $\Delta_{J}\setminus {\F}_a^{-1}(\Delta_{J})$ and the petal, we deduce that for $|c|$ sufficiently large, $\Phi^+(W)$ intersects the  horizontal line $v=c$. 
So $W\setminus \{P\}$, after quotienting by the boundary identification induced by $\F_a^{-1}$, is 
conformally bijective to a 
pair of neighbourhoods of the ends of $V_\theta^+/\langle w\to w-1\rangle$, that is to a pair of punctured discs. Hence 
$\Omega/\langle \F_a \rangle$ has a pair of puncture points (one either side of the parabolic axis) 
corresponding to $P$. \qed


\begin{remark}\label{bad-domains} If we were to choose $\Delta_J$ and $\Delta_{Cov}$ with 
boundaries approaching $Z=1$
tangentially to the parabolic axis, then the image under $\Phi^+$ of points of $\partial\Delta_J$ sufficiently close to $P$
might lie below some level $v=c$, in which case $(W\setminus\{P\})/\F_a$ would be an annulus rather than a punctured
disc and we would find that the new set $\Omega$ would differ from that in the case of a transverse intersection: a
horodisc at $Z=1$, together with the grand orbit of this horodisc, would be excised from the $\Omega$ of the transverse case. 
The representation of $C_2*C_3$ on 
${\mathbb H}$ would no longer be that of $PSL(2,\Z)$, but $\Lambda({\mathcal F}_a)$ would also be changed, 
by the addition of a countable union of discs, attached at the points of the grand orbit of $Z=1$. In Definition \ref{limitset} we required
$\Delta_J$ and $\Delta_{Cov}$ to have boundaries transverse to the parabolic axis, in order that the partition of
$\hat \C$ into $\Omega$ and $\Lambda$ be uniquely defined.
\end{remark}

\section{Proof of Theorem B}\label{thmb}

\subsection{Properties of the $2$-to-$1$ branch of ${\mathcal F}_a$ which fixes $\Lambda_{a,-}$}
For the proof of Theorem B we shall need to convert the branch of $\F_a$ which fixes $\Lambda_{a,-}$ into a parabolic-like map by quasiconformal surgery. The next two results set the scene.
Proposition \ref{dyn} ensures that this branch of $\F_a$ is locally holomorphic everywhere but on a neigbourhood of $S$
(the preimage of the parabolic fixed point). Proposition \ref{lune} ensures we have a sector at $S$ which can support the surgery that will turn the branch into a parabolic-like map.
\begin{prop}\label{dyn}
For every $a \in {\mathcal K}$, the restricted correspondence
$${\mathcal F}_a |:\ {\mathcal F}_a^{-1}(\Delta_{J}) \to \Delta_{J}$$
is a single-valued holomorphic map of degree two.

For each $Z\in \partial {\mathcal F}_a^{-1}(\Delta_{J})$, with the exception of $Z=S$, the pre-image of the parabolic fixed point $P$ other than $P$ itself,
there exists a neighbourhood of $Z$ on which  
${\mathcal F}_a |$ extends to a (single-valued) holomorphic map.

There exists a neighbourhood of $S$ on which ${\mathcal F}_a |$ extends locally to a $1$-to-$2$ holomorphic correspondence, the image of which is a neighbourhood of $P$. This correspondence between neighbourhoods of $S$ and $P$ is conformally conjugate to the $1$-to-$2$ correspondence $\zeta \to \pm\sqrt{\zeta}$ from the unit disc to itself.

\end{prop}

\begin{proof}
The fact that ${\mathcal F}_a |:\ {\mathcal F}_a^{-1}(\Delta_{J}) \to \Delta_{J}$ is a (single-valued) holomorphic map follows at once from Proposition 3.4, since an $n$-to-$1$ holomorphic correspondence defined on an open set is necessarily a holomorphic map. 

Moreover, given any $Z\in \partial {\mathcal F}_a^{-1}(\Delta_{J})$ which does not map to $P$ (the point $Z=1$), we may deform the boundary of $\Delta_{J}$ (without altering that of $\Delta_{Cov}$) in such a way that $Z$ now lies in the interior of the deformed $\Delta_{J}$, so the second statement also follows from Proposition 3.4.

Finally, a neighbourhood of $S$ ($Z=-2$) is 
mapped $1$-to-$2$  by $\F_a$ to a neighbourhood
of $P$ ($Z=1$), since $\F_a=J\circ Cov^Q_0$, and $Cov_0^Q:Z \rightarrow W$ is the $1:2$ correspondence which has formula 
$W=(-Z\pm \sqrt{12-3Z^2})/2$. The local conjugacy to $\zeta \to \pm \sqrt{\zeta}$ is immediate from the formula.
\qed
\end{proof}

%

\begin{prop}\label{lune}
 For every $a \in \mathcal K$ and Klein combination pair $(\Delta_{J},\Delta_{Cov})$ for $\F_a$, 
with boundaries transverse to the parabolic axis at the fixed point $P$, there exist a closed topological disc 
$V_a\subset \hat\C$ and angles $\theta_1=\theta_{a,1}>0$ and 
$\theta_2=\theta_{a,2}>0$, with $\theta_1+\theta_2<\pi$, with the following properties:
 \begin{enumerate}
 \item $\Lambda_{a,-}\subset V_a$ and $\Lambda_{a,-}\cap\partial V_a=\{P\}$; 
\item the boundary $\partial V_a$ of $V_a$ is smooth away from the parabolic fixed point $P$, where it meets $\partial \Delta_{J}$
at angles $\theta_1$ and $\theta_2$ (so at $P$ the boundary $\partial V_a$ has a `cone' of angle $\hat\theta =\pi - (\theta_1 + \theta_2)$);.
\item  $V'_a=\F_a^{-1}(V_a) \subset V_a$, and $\partial V_a' \cap \partial V_a = \{P\}$;
\item the boundary $\partial V_a'$ of $V_a'$ is smooth everywhere but at $P$, where it forms 
a cone of angle ${\hat \theta}$, and at the preimage $S$ of $P$, where it forms a cone of
angle $2\hat \theta$;
\item inside every neighbourhood of $P$ there exist $\F_a$-invariant arcs $\gamma_i:[0,1]\to {\bar V}_a$, $i=1,2$,
emanating from $P$ on the two sides of the parabolic axis, each $C^1$ and satisfying $\gamma_i(t)[1/2,1)\subset V_a\setminus V'_a$.
 \end{enumerate}
\end{prop}
\begin{proof}
 
By Proposition \ref{nice-domains}, for every $a \in \mathcal{K}$ we can choose  a Klein combination pair $(\Delta_{Cov},\Delta_{J})$ of
fundamental domains which have boundaries which are smooth at $P$ and transverse to the parabolic direction.
%
By Proposition \ref{domains}, 
$\F_a^{-1}(\overline{\Delta}_{J})\subset \overline{\Delta}_{J}$. We shall construct
$V_a$ by making a small change to the boundary of $\Delta_{J}$ in a neighbourhood of $P$, so that while
$V_a$ is not a fundamental domain for $J$ it retains the property that $\F_a^{-1}(V_a)\subset V_a$ and gains
the other properties listed.\\

Suppose firstly that $a\ne 7$, so we are in the `single petal' case.
Let $\ell$ denote $\partial \Delta_{J}$ and let
the angles at $P$ between $\ell$ and the parabolic axis be $\alpha_1$ and 
$\alpha_2=\pi - \alpha_1$. Choose $\theta$ such that $\mbox{max}(\alpha_1,\alpha_2)<\theta <\pi$ 
(such an angle $\theta$ exists since 
we started from a  Klein combination pair $(\Delta_{Cov},\Delta_{J})$  which have boundaries which are smooth at $P$ and transverse to the parabolic direction).
Let $U^+_{\theta}$ be a repelling petal containing an open sector of angle $2 \theta$ centered at $P$ 
given by Proposition \ref{petals}, and $\Phi^+: U^+_{\theta} \rightarrow V^+_{\theta}$ be a repelling Fatou coordinate (where $V_{\theta}$ is the subset of $\C$ consisting to all the $w=u+iv$ to the left of a curve which has 
asymptotes $u=-|v|\cot(\theta)-c$, with $c$ large).
As $\ell$ is $C^1$ at $P$, and so is $F_a^{-1}(\ell)$, with the same tangent at $P$, we know that for every point 
$R\in \ell$ sufficiently close to $P$ the open straight line segment (in whatever coordinate we are working in) from $R$ to 
$\F_a^{-1}(R)$ lies in $\Delta_{J}\setminus \F_a^{-1}(\Delta_{J})$. Thus we can foliate
the intersection $W$ between $\Delta_{J}\setminus \F^{-1}(\Delta_{J})$ and a suitable neighbourhood of $P$ by straight line segments. 
The set $W$ has two components, which we denote $W_1$ and $W_2$, one each side of the parabolic axis at $P$. Write $D_i$ ($i=1,2$) for $\bigcup_{n=0}^{\infty}\F_a^{-n}(W_i)$. The sets $D_i$ are foliated
by piecewise-linear leaves, each of which is invariant under $\F_a^{-1}$ and crosses each 
line $\F_a^{-n}(\ell)$ exactly once. In $\Phi^+(D_i)\subset V_\theta$ they become leaves invariant under $w\to w-1$. 
Consider a set of these leaves which are integer distances apart at the points where they meet $\Phi^+(\ell)$. Together 
with the lines $(\Phi^+(\F_a^{-n}(\ell)))_{n \geq 0}$ they create a `skew grid' in each of the 
$\Phi^+(D_i)$, $i=1,2$.\\


Choose $0<\theta_1<\alpha_1$ and  $0<\theta_2<\alpha_2$.
Using the skew grids as  coordinate systems, we can now construct in each $\Phi^+(D_i),\ i=1,2$, a 
smooth curve $m_i$ which at one end joins $\Phi^+(\ell)$ smoothly, 
at the other is asymptotic to a line at angle $\theta_i$ to the horizontal as $w$ tends to infinity, 
and in between crosses each leaf of 
the foliation exactly once, and each line $\Phi^+(\F_a^{-n}(\ell))$ exactly once. 
Note that the lines $(\Phi^+)^{-1}(m_i)$ lie outside $\Lambda_{a,-}$ since every point of 
$(\Phi^+)^{-1}(D_i)$ eventually leaves $\Delta_{J}$ under
some iterate of the branch of $\F_a$ which fixes $P$.
We now define $V_a$ to be the domain bounded by $\ell=\partial\Delta_{J}$ as modified by 
$(\Phi^+)^{-1}(m_1)$ and $(\Phi^+)^{-1}(m_2)$ in a neighbourhood of $P$, and define 
$V'_a$ to be $\F_a^{-1}(V_a)$.
The first three properties stated in the 
Proposition are immediate, and the 4th property follows from Proposition \ref{dyn}.
Finally, for the 5th property we note that each leaf of the foliation satisfies all the requirements 
except that it is piecewise-linear and not (in general) $C^1$. 
We rectify this
by replacing a chosen straight line leaf in each $W_i$, $i=1,2$, by a $C^1$ curve $n_i$ with the same end-points, say $R_i \in \ell$ and ${\mathcal F}_a^{-1}(R_i)$, such that $n_i$ meets $\ell$ and ${\mathcal F}_a^{-1}(\ell)$ at angles which sum to $\pi$: we then set $\gamma_i$ to be
$\bigcup_{n=0}^\infty \F_a^{-n}(n_i)$, parametrized appropriately.\\

In the case $a=7$ we have to modify the argument above to allow for the fact that we have 
three attracting petals and three
repelling petals. We omit details, but remark that the key difference is that whereas for $a\ne 7$ we can choose
$\theta_1$ and $\theta_2$ such that ${\hat \theta}=\pi-(\theta_1+\theta_2)$ is arbitrarily small,
for $a=7$, with the standard domains, by taking Fatou coordinates on appropriate overlapping attracting and 
repelling petals, one can show that $\theta_1$ and $\theta_2$ must satisfy $\hat \theta>2\pi/3 > 0$ but can be chosen with $\hat \theta$ arbitrarily close to $2\pi/3$.
\qed

\end{proof}


\subsection{Proof of Theorem B}\label{1}
By Proposition 5.1, for every $a \in {\mathcal K}$, and therefore in particular for every $a \in {\mathcal C}_\Gamma$, the 
correspondence ${\mathcal F}_a$, restricted to a neighbourhood of 
$\Lambda_{a,-}$, satisfies all the conditions necessary for it to be a 
{\it parabolic-like map} in the sense of \cite{L1}, except one: on a neighbourhood
of the point $S={\mathcal F}_a^{-1}(P)\setminus \{P\}$ it is not a single-valued map, as such a neighbourhood is mapped one-to-two 
onto a neighbourhood of $P$. However by redefining
${\mathcal F}_a$ on a `sector' at $P$ lying outside $\Lambda_{a,-}$, and adjusting the 
complex structure on this sector and its 
inverse images, we shall now modify a restriction of the branch of ${\mathcal F}_a$ fixing $\Lambda_{a,-}$, to yield 
a parabolic-like map $\widetilde \F$.\\

By Proposition \ref{lune}, at the parabolic fixed point $P$ the boundary $\partial V'$ of $V'$ forms a cone 
of angle $\hat \theta=\pi-(\theta_1+\theta_2)$, and at the 
preimage $S$ of $P$ it forms a cone of angle $2 \hat \theta$. 
Possibly by reducing $\theta_1,\theta_2$ and $\hat \theta$ we can choose $\epsilon >0$ small enough so that the round disc  
$D_2= \D(S, \epsilon)$ intersects $V'$ in a sector of angle $2 \hat \theta$, so that $D_1= \F(D_2)$ intersects
$V'$ in a sector of angle $\hat \theta$, and moreover $\partial V \cap \gamma_i \not \subset \overline D_1$
(where $V$ is the set, $\gamma_1,\gamma_2$ the invariant arcs, and $\hat \theta$ the angle given by Proposition \ref{lune}). 
Hence denoting by $\hat T_2$ the sector $(\pi-2 \theta_2,S,\pi+2 \theta_1)$ and by
 $\hat T_1$ the sector $(3\pi/2-\theta_2,P,\pi/2+ \theta_1)$, both $\hat T_1$ and $\hat T_2 $ are outside
$V'$,  and in particular
 $\hat T_2 \in V \setminus V'$.
 Set $\phi : D_2 \rightarrow \D$,  $\phi(z)= (z-S)/\epsilon)$,
and let $\psi: D_1 \rightarrow \D$ be the 
Riemann map sending $P$ to $0$. Then $\phi \circ \F^{-1} \circ \psi^{-1}$ is a degree $2$ proper and holomorphic map from the unit 
disc into itself, with a unique fixed point at $z=0$, and so pre- or post-composing with a rotation we can assume it to be
the map $P_0(z)=z^2$.\\

We are now going to modify $\F$ on $\hat T_2$ by quasiconformal surgery.
Lift to logarithm coordinates, and define the quasiconformal map  
$$G:  \{ x+iy\,\,|\,\,x<0,\,\, y \in [\pi,3\pi]\} \rightarrow \{
x+iy\,\,|\,\, x<0,\,\,y \in [0,2 \pi]\}$$ as follows:
$$ G(z)=\left\{
\begin{array}{cl}
2z-2\pi i&\mbox{on  } \{ x+iy\,\,|\,\,x<0,\,\, y \in [\pi, 3\pi/2-\theta_2]\}\\
\mbox{qc interpolation  } &\mbox{on  }  \{ x+iy\,\,|\,\,x<0,\,\,y \in [3\pi/2-\theta_2,5\pi/2+\theta_1]\}\\
2z-4\pi i&\mbox{on  } \{ x+iy\,\,|\,\,x<0,\,\, y \in [5\pi/2+\theta_1, 3 \pi]\} \\
\end{array}\right.
$$
Then the map $f= \phi^{-1} \circ exp  \circ G \circ log \circ \psi : D_1 \rightarrow D_2$ is also quasiconformal.
\begin{figure}
\centering
\psfrag{e}{\small $\phi$}
\psfrag{3}{\small $3\pi i$}
\psfrag{c}{\small $\psi$}
\psfrag{0}{\small $0$}
\psfrag{1}{}
\psfrag{2}{}
\psfrag{1'}{\small $(5\pi/2+\theta_1)i$}
\psfrag{2'}{\small $(3\pi/2-\theta_2)i$}
\psfrag{t}{}
\psfrag{p}{\small $\pi i$}
\psfrag{i}{\small $P_0$}
\psfrag{f}{\small $f$}
\psfrag{S}{\small $S$}
\psfrag{D2}{\tiny }

 \psfrag{P}{\small $2\pi i$}
 \psfrag{a}{\small $D_2$}
\psfrag{b}{\small $D_1$}
 \psfrag{7}{\small $3\pi/2- \theta_2$}
\psfrag{6}{\small $\pi/2 +\theta_1$}
\psfrag{A}{\small $\pi -2 \theta_2$}
\psfrag{B}{\small $\pi +2 \theta_1$}
\psfrag{T}{}
\psfrag{l}{\small $log$}
\psfrag{r}{\small $2z-4\pi i$}
\psfrag{s}{\small qc interpolation}
\psfrag{w}{\small $2z-2\pi i$}
\psfrag{y}{}
\psfrag{x}{}
\psfrag{y'}{\small $\hat T_1$}
\psfrag{x'}{\small $\hat T_2$}
	\includegraphics[width= 10cm]{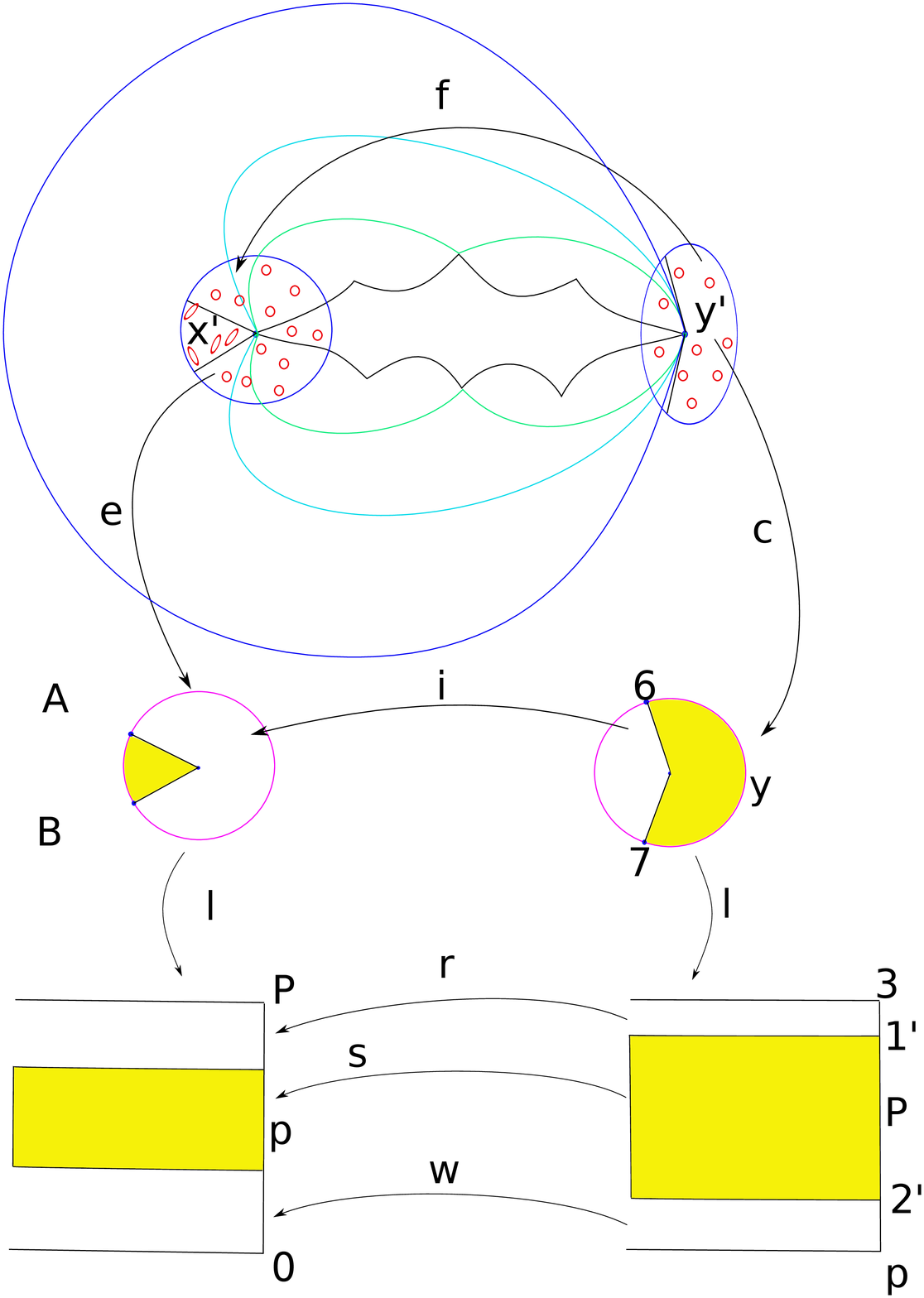}
\caption{\small surgery construction}
\label{C0}
\end{figure}
Define $U= V \cup D_1$, $U'=\F^{-1}(U)$, and the map 
$F : U' \rightarrow U$ to be:
$$ F =\left\{
\begin{array}{cl}
f^{-1} &\mbox{on  } D_2\\
\F &\mbox{on  } U' \setminus D_2 \\
\end{array}\right.
$$
The map $F : U' \rightarrow U$ is continuous, because it coincides with $\F$ everywhere but on 
$\hat T_2$,
and along the boundaries of $\hat T_2$ inside $D_2$ it is continuous by construction. So the map $F$
is quasiregular, proper of degree $2$, and holomorphic everywhere but on the sector
$\hat T_2$.

Setting $\tilde \mu = (f^{-1})^*(\mu_0)$, and spreading $\tilde \mu$ by the 
dynamics of $F$, we obtain on $U$ the Beltrami form:
$$ \overline \mu =\left\{
\begin{array}{cl}
\widetilde \mu &\mbox{on  } \hat T\\
(F^n)^*\widetilde \mu &\mbox{on  } F^{-n}(\hat T) \\
\mu_0 &\mbox{on  } U \setminus F^{-n}(\hat T) \\
\end{array}\right.
$$
Since the sector $\hat T_2$ lies outside $F^{-1}(V)$,
it follows that $F^{-i}(\hat T_2)$
lies outside $F^{-i-1}(V)$, and therefore the preimages of the sectors
$F^{-i}(\hat T_2)$ where we change the structure do not intersect each others. Hence the Beltrami form $\overline \mu$
is $F$-invariant, and by the Measurable Riemann Mapping Theorem there exists a quasiconformal map
$\varphi: U \rightarrow  \D$ such that $\varphi^*\mu_0=\overline \mu$.
Let us define 
$$\widetilde \F:=\varphi \circ F \circ \varphi^{-1}\,: \mathcal{V}'=\varphi(U') \rightarrow  \mathcal{V}=\varphi(U),$$
and set $\gamma_{+}= \varphi(\gamma_1) \cap \overline{\mathcal{V}}$ and  
$\gamma_{-}= \varphi(\gamma_2) \cap \overline{\mathcal{V}}$ (where $\gamma_1$ and $\gamma_2$ are the 
invariant arcs given by Proposition \ref{lune}).
Then $\gamma= \gamma_+ \cup \gamma_-$ is a dividing arc in the sense of 
Definition  \ref{parabolic-like-map},
and $(\widetilde \F,\mathcal{V}',\mathcal{V},\gamma)$ is a degree $2$ parabolic-like map, with filled Julia set 
$K=\varphi (\Lambda_-)$. The map $\widetilde \F$
is quasiconformally conjugate to $\F$ everywhere but on the sector $\hat T_2$ and its image, which do not
intersect the filled Julia set $K$. Moreover, this quasiconformal conjugacy is holomorphic everywhere but 
on the preimages of $\hat T_2$ (which do not
intersect the filled Julia set $K$). Therefore $\widetilde \F$ is hybrid conjugate to $\F$ on $K$. By the 
Straightening Theorem for parabolic-like maps (see \cite{L1}),
this implies that $\F$ is hybrid conjugate to a member of the family $Per_1(1)$ on $\Lambda_-$.
\qed

\end{document}